\documentclass[12pt, a4paper]{article}

\usepackage{graphicx}
\usepackage{amssymb}
\usepackage{amsmath}
\usepackage{amsthm}

\makeatletter
	
	\@addtoreset{equation}{section}
\makeatother

\newtheorem{thm}{Theorem}[section]
\newtheorem{prop}[thm]{Proposition}
\newtheorem{lem}[thm]{Lemma}
\newtheorem{cor}[thm]{Corollary}
\theoremstyle{definition}

\newtheorem{rem}[thm]{Remark}

\newcommand{\R}{\mathbb{R}}
\newcommand{\C}{\mathbb{C}}
\newcommand{\N}{\mathbb{N}}

\newcommand{\prob}[3]{\mathbb{P}^{#1}_{#2}\left( #3 \right)}
\newcommand{\mean}[3]{\mathbb{E}^{#1}_{#2}\left[ #3 \right]}
\newcommand{\sle}{\mathrm{SLE}}
\newcommand{\skle}{\mathrm{SKLE}}
\newcommand{\bmd}{\mathrm{BMD}}
\newcommand{\disk}{\mathbb{D}}
\newcommand{\uhp}{\mathbb{H}}
\newcommand{\Slit}{\mathsf{Slit}}
\newcommand{\slit}{\mathbf{s}}
\DeclareMathOperator{\dist}{dist}
\DeclareMathOperator{\res}{Res}
\DeclareMathOperator{\hcap}{hcap}

\title{On the slit motion obeying chordal Komatu--Loewner equation
with finite explosion time}
\author{
Takuya Murayama\thanks{Department of Mathematics, Kyoto University, Japan.} \footnote{Email:\ murayama@math.kyoto-u.ac.jp}
}
\date{}

\begin{document}
\maketitle
\begin{abstract}

\begin{sloppypar}
This paper studies the behavior of solutions near the explosion time
to the chordal Komatu--Loewner equation for slits,
motivated by the preceding studies by Bauer and Friedrich~(2008) and
by Chen and Fukushima~(2018).
The solution to this equation represents moving slits in the upper half-plane.
We show that the distance between the slits and driving function
converges to zero at its explosion time.
We also prove a probabilistic version of this asymptotic behavior
for stochastic Komatu--Loewner evolutions
under some natural assumptions.
\end{sloppypar}



\end{abstract}

\section{Introduction}
\label{sec:intro}

In the theory of conformal mappings on the complex plane,
it is often useful to consider the evolution
of a one-parameter family of conformal maps $\{g_t\}_{t\geq 0}$ or, equivalently,
regions $\{D_t\}_{t\geq 0}$ that are domains or ranges of these maps.
One of the main tools to describe such an evolution is
the \emph{Loewner differential equation},
from which some sharp estimates are obtained
on the Taylor coefficients of univalent functions,
such as Bieberbach's conjecture (de Branges' theorem).
See \cite{Po75} or \cite{Co95} for this direction.
These days, this equation is well known also in probability theory,
especially in the context of \emph{stochastic Loewner evolution (SLE)}
defined by Schramm~\cite{Sc00}.
This random process was introduced to find the scaling limits
of several two-dimensional discrete random processes on lattices,
and actually a lot of results have been established so far.

Basically, the Loewner equation concerns mappings on simply connected
planar domains, such as the unit disk $\disk$ (\emph{radial} case)
or upper half-plane $\uhp$ (\emph{chordal} case).
On the other hand, in the so-called \emph{bilateral} case,
Komatu~\cite{Ko43, Ko50} generalized this equation to a circular slit annulus,
an annulus with finitely many concentric circular slits removed.
On the basis of his argument, Bauer and Friedrich~\cite{BF04, BF06}
established a more detailed result in the radial case
and extended the radial SLE toward a circular slit disk.
The chordal case, on which we shall focus in this article,
was also generalized to a \emph{standard slit domain} of the form
$D=\uhp \setminus \bigcup_{j=1}^N C_j$, where $C_j$, $1 \leq j \leq N$,
are mutually disjoint horizontal slits (i.e., line segments parallel to the real axis)
by recent studies~\cite{BF08, CFR16, CF18}.
The resulting differential equation is called
the \emph{chordal Komatu--Loewner equation}~\cite{BF08, CFR16}.
In this case, the ranges of the conformal maps $\{g_t\}$ are specified
in terms of moving slits $\{C_j(t)\}$ whose dynamics is described
by the \emph{Komatu--Loewner equation for the slits}~\cite{BF08, CF18}.
See Figure~\ref{fig:hull}.

\begin{figure}
\centering
\includegraphics[width=0.9\columnwidth]{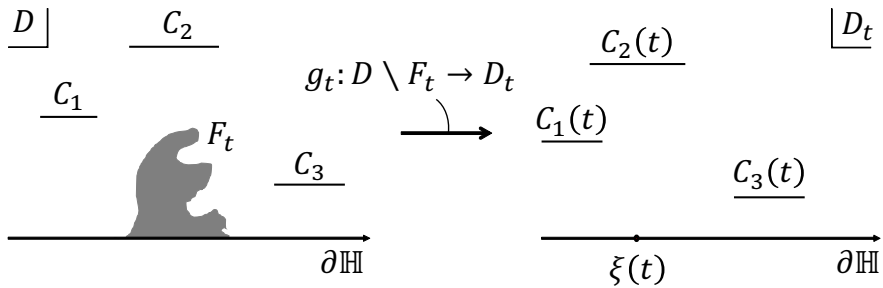}
\caption{Conformal maps and growing hulls}
\label{fig:hull}
\end{figure}

In the Loewner theory on simply connected domains,
this slit motion does not appear.
Thus, there are few results known on the behavior of the solution
to the Komatu--Loewner equation for the slits.
In particular, the explosion of this solution is a new obstacle of the theory.
Motivated by such a background,
we focus on the asymptotic behavior of the slit motion
around its explosion time $\zeta$ in this paper.
Assuming $\zeta<\infty$, we observe that the distance
between the slits and a moving point $\xi(t)$ on the real axis,
called the \emph{driving function} below, converges to zero as $t \to \zeta$.
Moreover, we prove a probabilistic version of this asymptotic behavior
for \emph{stochastic Komatu--Loewner evolutions},
which was introduced by Bauer and Friedrich~\cite{BF04, BF08} and
by Chen and Fukushima~\cite{CF18} to generalize SLE.

In order to provide a mathematical detail and
an appropriate intuition on the asymptotic behavior of the slits,
we now briefly recall the concrete form
of the chordal Komatu--Loewner equations.

Let us consider a typical case where $F_t$ in Figure~\ref{fig:hull}
is given by the trace $\gamma(0,t]$
of a simple curve $\gamma \colon [0, t_{\gamma}) \to \overline{D}$
satisfying $\gamma(0) \in \partial \uhp$ and $\gamma(0, t_{\gamma}) \subset D$.
Then for each $t \in [0, t_{\gamma})$, there exists a unique pair
of a standard slit domain $D_t$ and conformal map
$g_t \colon D \setminus \gamma(0,t] \to D_t$ with the hydrodynamic normalization
$g_t(z)=z+a_t/z+o(z^{-1})\; (z \to \infty)$.
The image $g_t(z)$ satisfies
the chordal Komatu--Loewner equation
\begin{equation} \label{eq:KL}
\frac{d}{dt}g_t(z) = -\pi \dot{a}_t \Psi_{D_t}(g_t(z),\xi(t)), \quad g_0(z)=z \in D,
\end{equation}
where $\dot{a}_t$ stands for the $t$-derivative of $a_t$.
The dynamics of the range $D_t$ is also described
by the Komatu--Loewner equation for the slits
\begin{equation} \label{eq:KLs}
\frac{d}{dt}z_j(t) = -\pi \dot{a}_t \Psi_{D_t}(z_j(t),\xi(t)),\quad
\frac{d}{dt}z^r_j(t) = -\pi \dot{a}_t \Psi_{D_t}(z^r_j(t),\xi(t)),
\end{equation}
where $z_j(t)$ (resp.\ $z^r_j(t)$) is the left (resp.\ right) endpoint
of the $j$-th slit $C_j(t)$ of $D_t$.
In the equations~\eqref{eq:KL} and \eqref{eq:KLs},
the driving function $\xi(t)$ is given
by $g_t(\gamma(t))=\lim_{z \to \gamma(t)}g_t(z) \in \partial \uhp$, and
the kernel $\Psi_{D_t}$ is the \emph{complex Poisson kernel}
of \emph{Brownian motion with darning} (BMD)
for the domain $D_t$~\cite[Lemma~4.1]{CFR16}.
If there are no slits (i.e., $D=\uhp$) and if $a_t=2t$ holds,
then the equation~\eqref{eq:KLs} does not appear,
and \eqref{eq:KL} reduces to the celebrated chordal Loewner equation
\begin{equation} \label{eq:Lo}
\frac{d}{dt}g_t(z)=\frac{2}{g_t(z)-\xi(t)}, \quad g_0(z)=z \in \uhp.
\end{equation}

In the previous paragraph, we start at a given trace $F_t=\gamma(0,t]$
and then obtain the driving function $\xi(t)$ and equations~\eqref{eq:KL}
and \eqref{eq:KLs}.
In turn, given a driving function $\xi \in C([0, \infty); \R)$,
we consider the initial value problem of \eqref{eq:KL} and \eqref{eq:KLs}.
In this case, let $[0 ,t_z)$ be the maximal time interval of existence
of a unique solution $g_t(z)$ to \eqref{eq:KL} for each $z \in D$.
Then it can be checked that the solutions $\{g_t(z); z \in D\}$ constitute
a conformal map $g_t \colon D \setminus F_t \to D_t$
hydrodynamically normalized,
where $F_t$ is given by $F_t:=\{z \in D; t_z \leq t\}$.
Though $F_t$ is not the trace of a simple curve in general,
it is at least a \emph{(compact $\uhp$-)hull} in $D$ as in Figure~\ref{fig:hull}.
Here, a hull means a non-empty, bounded and relatively closed subset of $\uhp$
whose complement in $\uhp$ is simply connected.
We call $\{g_t\}$ the (decreasing) Komatu--Loewner chain and
$\{F_t\}$ the Komatu--Loewner evolution driven by $\xi(t)$ in this article.
In particular, the stochastic Loewner evolution with parameter $\kappa>0$,
abbreviated as $\sle_{\kappa}$, is defined
by putting $\xi(t)=\sqrt{\kappa}B_t$ in the Loewner equation~\eqref{eq:Lo},
where $B_t$ is the one-dimensional standard Brownian motion.

In the no slit case $D=\uhp$, the Loewner evolution $\{F_t\}$ is defined
on the entire time interval $[0, \infty)$ if so is the driving function $\xi(t)$.
However, the Komatu--Loewner evolution $\{F_t\}$ is not necessarily defined
on $[0, \infty)$ even if $\xi(t)$ is defined there,
because the ranges $\{D_t\}$ in the right-hand side of \eqref{eq:KL}
is determined by the slit motion that solves \eqref{eq:KLs}.
Thus, $g_t$ and $F_t$ are defined only up to the explosion time $\zeta$
of the solution to \eqref{eq:KLs}.
This is a major difference between the Loewner and Komatu--Loewner equations,
and hence the explosion of the solution to \eqref{eq:KLs} is the main theme
of this paper as mentioned above.
In particular, our interests are the following two points:
\begin{itemize}
\item the asymptotic behavior of the slits $C_j(t)$ of $D_t$,
\item the relation between the asymptotic behaviors of $C_j(t)$ and of $F_t$.
\end{itemize}

To give a natural outlook on these two questions,
let us formally discuss some possibilities of finite time explosion.
The first possibility is a situation where $\{F_t\}$ touches or swallows
a certain slit $C_j$ at time $\zeta<\infty$.
Here, we say that $\{F_t\}$ \emph{swallows} a point $z \in \uhp$
if $z$ is not in the union $\bigcup_{t<\zeta}F_t$ but in a bounded component of
$\uhp \setminus \overline{\bigcup_{t<\zeta}F_t}$.
In this case, the unbounded component of
$D \setminus \overline{\bigcup_{t<\zeta}F_t}$ no longer has $N$ boundary slits.
Hence the equation~\eqref{eq:KLs} cannot have a solution
representing disjoint $N$ slits at $\zeta$.
The second one is the case where $F_t$ becomes unbounded in finite time.
This situation, however, does not seem to happen
if $\xi(t)$ is defined on the entire time interval $[0, \infty)$.
Since the `preimage' of $\xi(t)$ by $g_t$ is, loosely speaking, the `tip' of $F_t$,
the driving function $\xi(t)$ should diverge if $F_t$ becomes unbounded.
As a consequence, we are led to a guess
that only the former case occurs when $\zeta<\infty$
and that, if the slit $C_j$ is touched or swallowed by $F_t$,
then the corresponding slit $C_j(t)$ approaches $\xi(t)$.

We now state our main results that are based on our observations above.
Needless to say, it is difficult to verify all of these observations.
However, we can prove that
\begin{equation} \label{eq:char_intro}
\lim_{t \nearrow \zeta} \min_{1 \leq j \leq N} \dist(C_j(t), \xi(t))=0
\end{equation}
assuming that $\zeta<\infty$ (Theorem~\ref{thm:char_d}).
We note that \eqref{eq:char_intro} immediately implies that
$\lim_{t \nearrow \zeta} \Im z_j(t)=0$ for some $j$,
which justifies the comment in \cite[Theorem~4.1]{BF08}.
Moreover, we can establish the property~\eqref{eq:char_intro}
for the stochastic Komatu--Loewner evolution as well.
Let us recall that, motivated by \cite{BF08},
Chen and Fukushima~\cite{CF18} introduced $\skle_{\alpha, b}$
by the following stochastic differential equation (SDE) for the driving function:
\begin{equation} \label{eq:SKLE_dom}
d\xi(t)=\alpha(\xi(t), D_t)\,dB_t+b(\xi(t), D_t)\,dt.
\end{equation}
Under mild conditions on $\alpha$ and $b$,
the property~\eqref{eq:char_intro} still holds almost surely
for the solution to the system~\eqref{eq:KLs} and \eqref{eq:SKLE_dom}
(Theorem~\ref{thm:char_s}).
These two results, Theorems~\ref{thm:char_d} and \ref{thm:char_s},
are the main results of this paper.

In the proof of \eqref{eq:char_intro}, we need
to transform a Komatu--Loewner chain $\{g_t\}$ into a Loewner one $\{g^0_t\}$.
Such a transformation method was originally established
by Chen, Fukushima and Suzuki~\cite{CFS17}
and then generalized by the author~\cite{Mu18}.
See the paragraph after Theorem~\ref{thm:KLeq} in Section~\ref{sec:prel}
for the background on this transformation method.
In the paper~\cite{Mu18}, a version of Carath\'eodory's kernel theorem,
which is well known in complex analysis, was formulated and used extensively
to establish the general transformation method.
This kernel theorem will be used in the proof of \eqref{eq:char_intro} as well.

The rest of this paper is organized as follows:
Section~\ref{sec:prel} is devoted to a short review
on the previous results of \cite{CF18, Mu18}.
Section~\ref{sec:expl} is devoted
to the formulation and proof of the property~\eqref{eq:char_intro}.
We formulate \eqref{eq:char_intro} as Theorem~\ref{thm:char_d}
and its probabilistic version as Theorem~\ref{thm:char_s}
in Section~\ref{subsec:main}.
A key lemma, Lemma~\ref{lem:condB}, is also established in the same subsection.
Then we prove Theorem~\ref{thm:char_d} through Sections~\ref{subsec:outline},
\ref{subsec:claim_hcap} and \ref{subsec:claim_drive}.
The proof of Theorem~\ref{thm:char_s} is given in Section~\ref{subsec:char_s}
based on the proof of Theorem~\ref{thm:char_d}.

\section{Preliminaries}
\label{sec:prel}

Let $A_1$, ..., $A_N$ be disjoint compact continua in $\uhp$.
Here, by a continuum we mean a connected closed sets in $\C$
having more than one point.
We work on a domain of the form $D:=\uhp \setminus \bigcup_{j=1}^N A_j$
thoughout this paper.
A basic fact is that, for any hull (or empty set) $F \subset D$,
the \emph{canonical map} $f_F \colon D \setminus F \to \tilde{D}$ exists
by \cite[Proposition~2.3]{Mu18}.
This means that $f_F$ is a conformal map onto a standard slit domain $\tilde{D}$
with the hydrodynamic normalization $\lim_{z \to \infty}(f_F(z)-z)=0$,
and that the pair $(f_F, \tilde{D})$ is unique.
After taking Schwarz's reflection, the canonical map $f_F$
has the Laurent expansion
\[
f_F(z)=z+\frac{\hcap^D(F)}{z}+o(z^{-1})\quad \text{as $z \to \infty$.}
\]
The positive constant $\hcap^D(F)$ is called
the \emph{half-plane capacity of $F$ relative to $D$}.

Another basic fact that is used later
is a variant of Carath\'eodory's \emph{kernel theorem}.
For a sequence of subdomains $D_n$ of $\uhp$,
we define the \emph{kernel} of $\{D_n\}$ \cite[Definition~3.7]{Mu18}
as the largest unbounded domain such that
its every compact subset is included by $D_n$ for all sufficiently large $n$.
Under the assumption that
\begin{enumerate}
\item[(K.1)] all $D_n$ contain $\uhp \cap \Delta(0, L)$ for some fixed $L>0$,
\end{enumerate}
the kernel exists uniquely.
Here $\Delta(a,r):=\{z \in \C; \lvert z-a \rvert >r\}$ for $a \in \C$ and $r>0$.
We say that $\{D_n\}$ converges to its kernel
in the sense of \emph{kernel convergence}
if all subsequences of $\{D_n\}$ have the same kernel.
We consider, on such domains $D_n$, a sequence of univalent functions
$f_n \colon D_n \to \uhp$ such that
\begin{enumerate}
\item[(K.2)] $\lim_{z \to \infty} (f_n(z)-z)=0$;
\item[(K.3)] $\lim_{z \to \xi_0} \Im f_n(z)=0$
for all $\xi_0 \in \partial \uhp \cap \Delta(0, L)$.
\end{enumerate}

\begin{lem}[{\cite[Lemma~3.9]{Mu18}}] \label{lem:K1forRan}
Under Assumptions {\rm (K.1)--(K.3)}, the sequence of the ranges
$\tilde{D}_n:=f_n(D_n)$ also satisfies Condition {\rm (K.1)}
with the constant $L$ in {\rm (K.1)} replaced by $2L$.
\end{lem}

\begin{thm}[{\cite[Theorem~3.8]{Mu18}}] \label{thm:kernel}
Suppose that, under Assumptions {\rm (K.1)--(K.3)},
the sequence $\{D_n\}$ converges to a domain
$D=\uhp \setminus \bigcup_{j=0}^N A_j$ in the sense of kernel convergence,
where $A_0$ is a hull or an empty set,
each $A_j$ for $1 \leq j \leq N$ is a connected compact subset
whose complement in $\uhp$ is simply connected,
and all $A_j$'s are disjoint.
Then the following are equivalent:
\begin{enumerate}
\item \label{cond:fnconv}
$\{f_n\}$ converges to a univalent function
$f \colon D \to \uhp$ locally uniformly on $D$;
\item \label{cond:ranconv}
$\{\tilde{D}_n\}$ converges to a domain $\tilde{D}$
in the sense of kernel convergence.
\end{enumerate}
If one of these holds, then $\tilde{D}=f(D)$
and $f_n^{-1} \to f^{-1}$ locally uniformly on $D$.
\end{thm}

Note that the locally uniform convergence of $\{f_n\}$ makes sense
since every compact subset of $D$ is eventually included by $D_n$.
(In \cite{Mu18} the abbreviation `u.c.'\ is used to indicate
``uniform convergence on compacta'' following \cite{Co95},
but in this paper we avoid using it for the sake of readability.)

Keeping these two basic facts in mind,
we proceed to the correspondence between driving functions
and families of continuously growing hulls via the Komatu--Loewner equations,
which was established in \cite{CF18, Mu18}.
We regard \eqref{eq:KLs} as an ordinary differential equation (ODE)
on the open subset
\begin{align*}
\Slit=&\{\slit=(\slit_l)_{l=1}^{3N}
=(y_1, \ldots, y_N, x_1, \ldots, x_N, x^r_1, \ldots, x^r_N) \in \R^{3N} \\
&; y_j>0, x_j<x^r_j, \ \text{either $x^r_j<x_k$ or $x^r_k<x_j$
whenever $y_j=y_k$, $j \neq k$}\}
\end{align*}
of $\R^{3N}$ as follows:
For a vector $\slit=(y_1, \ldots, y_N, x_1, \ldots, x_N, x^r_1, \ldots, x^r_N) \in \Slit$,
the segment is denoted by $C_j(\slit)$ whose endpoints are $z_j:=x_j+iy_j$
and $z^r_j:=x^r_j+iy_j$.
We also put $D(\slit):=\uhp \setminus \bigcup_{j=1}^N C_j(\slit)$.
The functions
\[
b_l(\xi_0, \slit)=\begin{cases}
-2\pi \Im\Psi_{D(\slit)}(z_l, \xi_0) &(1 \leq l \leq N) \\
-2\pi \Re\Psi_{D(\slit)}(z_{l-N}, \xi_0) &(N+1 \leq l \leq 2N) \\
-2\pi \Re\Psi_{D(\slit)}(z^r_{l-2N}, \xi_0) &(2N+1 \leq l \leq 3N)
\end{cases}
\]
are locally Lipschitz continuous on $\R \times \Slit$ by \cite[Lemma~4.1]{CF18}
(see also \cite[Section~2.2]{Mu18}).
By utilizing these notations, we can write \eqref{eq:KLs} in the form
\begin{equation} \label{eq:gKLs}
\frac{d}{dt}\slit_l(t)=\frac{\dot{a}_t}{2}b_l(\xi(t), \slit(t)),\quad 1 \leq l \leq 3N.
\end{equation}

Let $\xi(t)$ be a continuous function on a fixed interval $[0, t_0)$
and $a_t$ be a strictly increasing and differentiable function on this interval
with $a_0=0$.
Since the right-hand side of \eqref{eq:gKLs} satisfies the local Lipschitz condition,
there exists a unique solution $\slit(t)$ with arbitrary initial value in $\Slit$
up to its explosion time $\zeta$.
The time $\zeta$ may be strictly less than $t_0$.
For this solution $\slit(t)$, the equation \eqref{eq:KL} is written as
\begin{equation} \label{eq:gKL}
\frac{d}{dt}g_{t}(z) = -\pi \dot{a}_t \Psi_{\slit(t)}(g_{t}(z),\xi(t)),
\quad g_{0}(z)=z \in D(\slit(0)),
\end{equation}
where we put $\Psi_{\slit}:=\Psi_{D(\slit)}$ for $\slit \in \Slit$.
For each point $z$ in $D:=D(\slit(0))$, this equation has a unique solution $g_t(z)$
up to the time $t_z:=\zeta \wedge \sup\{t>0; \lvert g_t(z)-\xi(t) \rvert >0\}$
by \cite[Theorem~5.5 (i)]{CF18}.
The sets $F_t:=\{z \in D ; t_z \leq t\}$, $t<\zeta$, constitute
a family of growing (i.e., strictly increasing) hulls in $D$, and
the function $g_t \colon D \setminus F_t \to D(\slit(t))$
is the canonical map for $F_t$.
See \cite[Section~5]{CF18} for further detail.

While we have seen in Section~\ref{sec:intro} that
the Komatu--Loewner equations were obtained
for the canonical map induced from a simple curve,
we now notice that these equations should be established
even if we start at a nice family of growing hulls.
To explain this fact precisely,
let $\{F_t\}_{t \in [0, t_0)}$ be a family of growing hulls
in a standard slit domain $D$ (of $N$ slits)
and $g_t \colon D \setminus F_t \to D_t$ be the canonical map.
We say that
\begin{itemize}
\item $\{F_t\}$ is \emph{continuous}
if $\{D \setminus F_t\}$ is continuous in the sense of kernel convergence
\cite[Definition~4.2]{Mu18};
\item $\xi \colon [0, t_0) \to \R$ is the \emph{driving function} of $\{F_t\}$
if, for each $t \in [0, t_0)$,
\begin{equation} \label{eq:shrink}
\bigcap_{\delta>0}\overline{g_t(F_{t+\delta} \setminus F_t)}=\{\xi(t)\}.
\end{equation}
\end{itemize}
Suppose that $\{F_t\}$ is continuous.
Then the range $\{D_t\}$ is continuous in the sense of kernel convergence
by \cite[Lemma~4.4]{Mu18}.
The half-plane capacity $\hcap^D(F_t)$ is also continuous and strictly increasing.
Hence we can take a continuous $\Slit$-valued function $\slit(t)$
satisfying $D_t=D(\slit(t))$
and reparametrize $\{F_t\}$ so that $\hcap^D(F_t)$ is differentiable in $t$.
The next theorem is, therefore, general enough to discuss
what kind of hulls induce the canonical maps
satisfying the chordal Komatu--Loewner equations.

\begin{thm}[{\cite[Theorem~4.6]{Mu18}}] \label{thm:KLeq}
Let $a_t \in C^1([0, t_0); \R)$ be strictly increasing with $a_0=0$ and
$\xi(t) \in C([0, t_0); \R)$.
The following are equivalent:
\begin{enumerate}
\item \label{cond:contihull}
$\{F_t\}_{t \in [0, t_0)}$ is a family of continuously growing hulls in $D$,
its driving function is $\xi(t)$, and $\hcap^D(F_t)=a_t$.
\item \label{cond:KLeqs}
The slits $\slit(t)$ and map $g_t(z)$ solve \eqref{eq:gKLs} and \eqref{eq:gKL}
with $\zeta \geq t_0$.
\end{enumerate}
\end{thm}

The condition \eqref{cond:contihull} in Theorem~\ref{thm:KLeq}
is stable under conformal transformation.
More precisely, let $V$ be a subdomain of $D$
with $\bigcup_{t \in [0, t_0)} F_t \subset V$,
$\tilde{D}$ be another slit domain of a possibly different number of slits,
and $h$ be a univalent function from $V$ into $\tilde{D}$.
By \cite[Theorem~4.8]{Mu18}, the family of the images $\{h(F_t)\}_{t \in [0, t_0)}$
by $h$ is again a family of continuously growing hulls in $\tilde{D}$.
Let $\tilde{g}_t$ be the canonical map for $h(F_t)$
and $h_t:=\tilde{g}_t \circ h \circ g_t^{-1}$.
The driving function of $\{h(F_t)\}$ is $h_t(\xi(t))$,
and $(d/dt)\hcap^{\tilde{D}}(h(F_t))=h_t'(\xi(t))^2\dot{a}_t$ holds.
The case where $\tilde{D}=\uhp$ was examined in \cite{CFS17}
to reduce the analysis of SKLE to that of SLE.
After that, the general case was proven in \cite{Mu18}
to give a full comprehension of the locality of $\skle_{\sqrt{6}, -b_{\bmd}}$
\cite[Theorem~4.9]{Mu18}
and to investigate the existing SLE-type processes on multiply connected domains
via the Komatu--Loewner equations.
The study~\cite{Mu19+b} on the relation between $\skle_{\alpha, \beta}$
and the \emph{Laplacian-$b$ motion}~\cite{La06}
illustrates the latter motivation well.

Finally, we note that all the results summarized in this section
are also the case for the chordal Loewner equation on $\uhp$
by defining $\Psi_{\uhp}(z, \xi_0)=\pi^{-1}(z-\xi_0)^{-1}$
except that we do not need to consider the equation for the slits.

\begin{rem}
$\Im \Psi_D(z, \xi_0)=K^*_D(z, \xi_0)$ is always positive
because $K^*_D$ is the Poisson kernel of BMD.
Hence both \eqref{eq:gKLs} and \eqref{eq:gKL} yield downward flows,
and the hull $F_t$ consists of the points $z$
whose images $g_t(z)$ eventually reach the point $\xi(t)$ on $\partial \uhp$.
The flow of $g_t(z)$ and the continuity of $\xi(t)$
thus strongly affect the shape of $F_t$.
As for the chordal Loewner equation in $\uhp$,
visual and detailed expositions on this relation can be found
in some literature, for example, in \cite[Chapter~2]{Ka15}.
Our observations in Section~\ref{sec:intro} also
comes from such a visual comprehension.
\end{rem}

\section{Main results and proof}
\label{sec:expl}

\subsection{Asymptotic behavior of the slit motion}
\label{subsec:main}

Thoughout this section, we fix a standard slit domain $D$ of $N (\geq 1)$ slits.
The ODEs \eqref{eq:gKLs} and \eqref{eq:gKL} under
the \emph{half-plane capacity parametrization} $a_t=2t$ are written as follows:
\begin{gather}
\frac{d}{dt}\slit_l(t)=b_l(\xi(t), \slit(t)), \quad 1 \leq l \leq 3N,
\label{eq:KLsvec} \\
\frac{d}{dt}g_{t}(z) = -2\pi \Psi_{\slit(t)}(g_{t}(z),\xi(t)),\quad g_{0}(z)=z \in D(\slit(0)).
\label{eq:KLvec}
\end{gather}
We formulate the asymptotic behavior~\eqref{eq:char_intro}
in terms of the slit vector $\slit(t)$.
We define a function $R(\xi_0, \slit)$ on $\R \times \Slit$ by
\[
R(\xi_0, \slit):=\min_{1 \leq j \leq N}\dist(C_j(\slit), \xi_0).
\]
This function is clearly \emph{invariant under horizontal translation}, that is,
$R(\xi_0, \slit)=R(0, \slit-\widehat{\xi_0})$.
Here, $\widehat{\xi_0} \in \R^{3N}$ stands for the vector
whose first $N$ entries are zero and last $2N$ entries are $\xi_0$.
The functions $b_l$ on the right-hand side of \eqref{eq:KLsvec}
are also invariant under horizontal translation by \cite[Eq.~(3.29)]{CF18}.
For later use, we adopt the notation $f(\slit):=f(0, \slit)$
when a function $f$ on $\R \times \Slit$ has this invariance.
We have, for example, $f(\xi_0, \slit)=f(\slit-\widehat{\xi_0})$ under this notation.
The main result in this section is now stated as follows:

\begin{thm} \label{thm:char_d}
Suppose that $\xi \in C([0, \infty); \R)$ and
$\slit^{\mathrm{int}} \in \Slit$ with $D(\slit^{\mathrm{int}})=D$ are given.
Let $\zeta$ denote the explosion time of the solution $\slit(t)$
to \eqref{eq:KLsvec} driven by $\xi(t)$ with initial value $\slit^{\mathrm{int}}$.
\begin{enumerate}
\item \label{thm:char_dI}
If $\zeta$ is finite, then it holds that
\begin{equation} \label{eq:charexpl}
\lim_{t \nearrow \zeta} R(\xi(t), \slit(t))=0.
\end{equation}
\item \label{thm:char_dII}
The inequality $\zeta \geq 2y_0^2$ holds,
where $y_0:=\min_{1 \leq l \leq N}\slit^{\mathrm{int}}_l$.
\end{enumerate}
\end{thm}

In the proof of this and the next theorems,
we consider the following condition for a function $f$ on $\R \times \Slit$:
\begin{enumerate}
\item[(B)]
$f(\xi_0, \slit)$ is bounded on the set
$\{(\xi_0, \slit) \in \R \times \Slit; R(\xi_0, \slit) > r\}$ for each $r>0$.
\end{enumerate}
If $f$ is invariant under horizontal translation,
then this condition is equivalent to the one that
\begin{enumerate}
\item[(B')] $f(\slit)=f(0, \slit)$ is bounded on the set
$\{\slit \in \Slit; R(\slit)=R(0, \slit) > r\}$ for each $r>0$.
\end{enumerate}
Since most of the functions appearing in this paper
is invariant under horizontal translation,
the latter form (B') is more convenient to our argument.
We shall observe in Lemma~\ref{lem:condB}
that the functions $b_l$ enjoy Condition (B).

We now provide a probabilistic version of Theorem~\ref{thm:char_d}.
Let $\alpha$ be a non-negative function on $\Slit$ homogeneous with degree $0$
and $b$ be a function on the same space homogeneous with degree $-1$,
both of which are supposed to enjoy the local Lipschitz continuity.
Here, a function $f(\slit)$ of $\slit \in \Slit$ is said to be
\emph{homogeneous with degree $\delta \in \R$}
if $f(c\slit)=c^{\delta}f(\slit)$ holds for all $c>0$ and $\slit \in \Slit$.
The stochastic Komatu--Loewner evolution
$\skle_{\alpha, b}$~\cite[Section~5.1]{CF18} is defined as
the family $\{F_t\}$ of continuously random growing hulls in $D$
produced by \eqref{eq:KLvec} and the system of SDEs \eqref{eq:KLsvec} and
\begin{equation} \label{eq:SKLE}
d\xi(t)=\alpha(\xi(t), \slit(t))\,dB_t+b(\xi(t), \slit(t))\,dt.
\end{equation}
Here, $(B_t)_{t \geq 0}$ is a one-dimensional standard Brownian motion,
and the coefficients in \eqref{eq:SKLE} are defined by
\[
\alpha(\xi_0, \slit):=\alpha(\slit-\widehat{\xi_0})\; \text{and}\;
b(\xi_0, \slit):=b(\slit-\widehat{\xi_0}).
\]
By definition, these coefficients are invariant under horizontal translation.
Since the local Lipschitz condition is assumed,
the system of SDEs~\eqref{eq:KLsvec} and \eqref{eq:SKLE}
has a unique strong solution that may blow up \cite[Theorem~4.2]{CF18}.

\begin{thm} \label{thm:char_s}
Suppose that functions $\alpha \geq 0$ and $b$ on $\Slit$ satisfy
the local Lipschitz continuity, homogeneity with degree $0$ and $-1$, respectively,
and Condition~{\rm (B)}.
For $\xi^{\mathrm{int}} \in \R$ and $\slit^{\mathrm{int}} \in \Slit$
with $D(\slit^{\mathrm{int}})=D$,
let $\zeta$ be the explosion time of the solution $W_t=(\xi(t), \slit(t))$ to the SDEs
\eqref{eq:KLsvec} and \eqref{eq:SKLE} with initial value $\mathbf{w}^{\mathrm{int}}=(\xi^{\mathrm{int}}, \slit^{\mathrm{int}})$.
\begin{enumerate}
\item \label{thm:char_sI}
The property \eqref{eq:charexpl} holds almost surely
on the event $\{\zeta<\infty\}$.
\item \label{thm:char_sII}
The inequality $\zeta \geq 2y_0^2$ holds almost surely,
where $y_0:=\min_{1 \leq l \leq N}\slit^{\mathrm{int}}_l$.
\end{enumerate}
\end{thm}

We shall discuss a non-trivial example that satisfies the assumptions
of Theorem~\ref{thm:char_s} in the forthcoming paper~\cite{Mu19+b}.

\begin{rem}
We have assumed that the driving function $\xi(t)$ is defined
on the infinite time interval $[0, \infty)$ in Theorem~\ref{thm:char_d}.
Although we have not assumed it in Theorem~\ref{thm:char_s},
we shall observe in Section~\ref{subsec:char_s} that,
under Condition (B), the process $\xi(t)$ can be extended continuously
as long as the slit vector $\slit(t)$ does not blow up.
If we consider the situation where $\xi(t)$ diverges in finite time,
then the conclusion of Theorem~\ref{thm:char_d} may change
as discussed in Section~\ref{sec:intro}.
We cannot tell whether the condition \eqref{eq:charexpl} holds or not,
and the hulls $\{F_t\}$ may ``creep along to infinity very close to the real axis''
\cite[Section~5.1]{BF08} in this case.
\end{rem}

The proof of Theorems~\ref{thm:char_d} and \ref{thm:char_s} takes several steps.
The first one is to prove the following key lemma on Condition (B):

\begin{lem} \label{lem:condB}
The functions $b_l(\xi_0, \slit)$, $1 \leq l \leq 3N$, and
\[
b_{\bmd}(\xi_0, \slit):=2\pi \lim_{z \to \xi_0}
\left(\Psi_{\slit}(z, \xi_0)+\frac{1}{\pi}\frac{1}{z-\xi_0}\right)
\]
all satisfy Condition {\rm (B)}.
\end{lem}

We call $b_{\bmd}(\slit):=b_{\bmd}(0, \slit)$
the \emph{BMD domain constant} of the domain $D(\slit)$~\cite[Section~6.1]{CF18}.
By \cite[Lemma~6.1]{CF18},
$b_{\bmd}$ is invariant under horizontal translation,
homogeneous with degree $-1$ and locally Lipschitz continuous.

\begin{proof}[Proof of Lemma~\ref{lem:condB}]
It suffices to prove Condition (B') as mentioned above.
We use classical estimates on the family
\[
S:=\{f \colon \disk \to \C; \text{$f$ is univalent on $\disk$},
f(0)=0\; \text{and}\; f'(0)=1\},
\]
where $\disk$ stands for the unit disk centered at the origin.

Recall from \cite[Section~2.1]{Mu18} that the function
\[
\mathbf{H}_{\slit}(z, \xi_0)=\Psi_{\slit}(z, \xi_0)+\frac{1}{\pi}\frac{1}{z-\xi_0},\quad
z \in D(\slit), \xi_0 \in \R,
\]
defined in \cite[Eq.~(2.2)]{Mu18} is holomorphic in
$z \in D(\slit) \cup \Pi D(\slit) \cup \partial \uhp$
after taking Schwarz's reflection.
Here, $\Pi$ stands for the mirror reflection with respect to the real axis.
It follows from definition that
$b_{\bmd}(\xi_0, \slit)=2\pi \mathbf{H}_{\slit}(\xi_0, \xi_0)$.
Accordingly we can check by using \cite[Theorem~11.2]{CFR16} that
$\Psi_{\slit}(z, \xi_0)$ defines a conformal map
from $D(\slit) \cup \Pi D(\slit) \cup \partial \uhp \cup \{\infty\}$
onto $\tilde{D} \cup \Pi \tilde{D} \cup \partial \uhp \cup \{\infty\}$,
where $\tilde{D}$ is another standard slit domain.
Its Laurent expansion around $\xi_0$ is given by
\begin{align}
\Psi_{\slit}(z, \xi_0)
&=-\frac{1}{\pi}\frac{1}{z-\xi_0}+\frac{1}{2\pi}b_{\bmd}(\xi_0, \slit)
+\left(\mathbf{H}_{\slit}(z, \xi_0)-\frac{1}{2\pi}b_{\bmd}(\xi_0, \slit)\right) \notag \\
&=-\frac{1}{\pi}\frac{1}{z-\xi_0}+\frac{1}{2\pi}b_{\bmd}(\xi_0, \slit)+o(1),\quad
z \to \xi_0. \label{eq:cPk}
\end{align}

Now assume that $R(\slit)=R(0, \slit)>r$ for some $r>0$.
Since $T(z):=-1/z$ is a linear fractional transformation
that maps $0$ to $\infty$ and $\infty$ to $0$,
the function $h(z):=(\pi r)^{-1}(T \circ \Psi_{\slit})(rz, 0)$ is univalent on $\disk$.
By the expansion~\eqref{eq:cPk}, we have
\begin{align*}
h(z)&=-\frac{1}{\pi r} \left(-\frac{1}{\pi rz} + \frac{1}{2\pi}b_{\bmd}(\slit)
+ o(1)\right)^{-1}\\
&=z \cdot \frac{1}{1-(r/2)b_{\bmd}(\slit)z+o(z)} \\
&=z+\frac{r}{2}b_{\bmd}(\slit)z^2+o(z^2),\quad z \to 0,
\end{align*}
which yields $h \in S$.
Thus, we can apply Bieberbach's theorem (see e.g.\ \cite[Theorem~1.5]{Po75}
or \cite[Theorem~14.7.7]{Co95}) to $h$ to obtain
\[
\lvert h''(0)\rvert \leq 2\cdot 2!, \quad\text{i.e.,}\quad
\lvert b_{\bmd}(\slit) \rvert \leq \frac{4}{r}.
\]

To show Condition (B') for $b_l$, we use Koebe's one-quarter theorem:
\begin{equation} \label{eq:Koebe}
f(\disk) \supset B(0, 1/4) \quad \text{for any $f \in S$}.
\end{equation}
See \cite[Corollary~1.4]{Po75} or \cite[Theorem~14.7.8]{Co95}
for the proof of this theorem.
We now observe from \eqref{eq:Koebe} that the univalent function $h$
maps a region outside $\disk$ into $\overline{\Delta(0, 1/4)}$.
Since $\Psi_{\slit}(z, 0)=(\pi r)^{-1}(T \circ h)(z/r)$ holds, we have
\begin{equation} \label{eq:inclBMD}
\Psi_{\slit}(D(\slit) \cap \overline{\Delta(0, r)}, 0)
\subset \overline{B(0, (4\pi r)^{-1})}.
\end{equation}
In particular, for the endpoints $z_j$ and $z^r_j$ of the slit $C_j(\slit)$, we get
\begin{equation} \label{eq:boundBMD}
\lvert \Psi_{\slit}(z_j, 0) \rvert \vee \lvert \Psi_{\slit}(z^r_j, 0) \rvert
\leq \frac{1}{4\pi r}, \quad 1 \leq j \leq N.
\end{equation}
This implies Condition (B') for $b_l$.
\end{proof}

Theorems~\ref{thm:char_d}~\eqref{thm:char_dII}
and \ref{thm:char_s}~\eqref{thm:char_sII} easily follow
from the estimate \eqref{eq:boundBMD} in the above proof.
We prove only the former here,
since the latter is obtained in a quite similar way.

\begin{proof}[Proof of Theorem~\ref{thm:char_d}~\eqref{thm:char_dII}]
By \eqref{eq:boundBMD} we have
\[
0 < 2\pi \Im \Psi_{\slit(t)}(z_j(t), \xi(t)) \leq \frac{1}{2 R(\xi(t), \slit(t))}
\leq \frac{1}{2 \min_{1 \leq l \leq N}\slit_l(t)}.
\]
We thus see from \eqref{eq:KLsvec} and the definition of $b_l$
that none of the $\slit_l(t)$'s goes to zero before $Y(t)$ goes to zero,
where $Y(t)$ is the solution to the ODE
\[
\frac{dY(t)}{dt}=-\frac{1}{2Y(t)},\quad Y(0)=y_0.
\]
It is easy to check that $Y(t)$ satisfies $t=2(y_0^2-Y(t)^2)$.
Hence Theorem~\ref{thm:char_d}~\eqref{thm:char_dII} follows by letting $Y(t) \to 0$.
\end{proof}

\subsection{Outline of the proof of Theorem~\ref{thm:char_d}~\eqref{thm:char_dI}}
\label{subsec:outline}

Suppose that $\xi \in C([0, \infty); \R)$
and $\slit^{\mathrm{int}} \in \Slit$ with $D(\slit^{\mathrm{int}})=D$ are given,
and let $\zeta$ denote the explosion time of the solution $\slit(t)$
to \eqref{eq:KLsvec} driven by $\xi(t)$ with initial value $\slit^{\mathrm{int}}$.
Moreover we suppose that $\zeta$ is finite.

\begin{prop} \label{prop:inf_sup}
If
\begin{equation} \label{eq:liminf}
\liminf_{t \nearrow \zeta} R(\xi(t), \slit(t)) = 0
\end{equation}
holds, then \eqref{eq:charexpl} holds.
\end{prop}

\begin{proof}
Suppose that \eqref{eq:liminf} holds but
$\limsup_{t \nearrow \zeta} R(\xi(t), \slit(t)) \geq 5r$ for some $r>0$.
There are then two increasing sequences $\{t_n\}_{n=1}^{\infty}$
and $\{t'_n\}_{n=1}^{\infty}$ both converging to $\zeta$ such that
$R(\xi(t_n), \slit(t_n))>4r$ and $R(\xi(t'_n), \slit(t'_n)) \leq r$.
Taking their subsequences if necessary, we may and do assume
$t_n<t'_n<t_{n+1}$ for $n \in \N$ without loss of generality.
By this assumption $\lim_{n \to \infty}(t'_n-t_n)=\zeta-\zeta=0$,
but, in fact, we can show $\inf_n \lvert t'_n-t_n \rvert >0$ as follows:
Let
\begin{align*}
\delta&:=\min \{\lvert s-t \rvert;
s, t \in [0, \zeta], \lvert \xi(s)-\xi(t) \rvert \geq r\}, \\
M&:=\sup \{\lvert b_l(\xi_0, \slit) \rvert;
(\xi_0, \slit) \in \R \times \Slit, R(\xi_0, \slit)>2r\}, \\
I&:=\{t \in [0, \zeta); R(\xi(t), \slit(t))>2r\}.
\end{align*}
The constant $\delta$ is positive, $M$ is finite by Lemma~\ref{lem:condB}, and
\[
\max_{1 \leq l \leq 3N} \lvert \slit_l(s)-\slit_l(t) \rvert \leq M \lvert s-t \rvert
\]
by \eqref{eq:KLsvec} if $s$ and $t$ belong to the same subinterval of $I$.
Thus it follows from the definition of $\{t_n\}_n$ and $\{t'_n\}_n$ that
\[
\lvert t'_n-t_n \rvert \geq \delta \wedge \frac{r}{M}.
\]
Since the right-hand side is independent of $n$,
we have $\inf_n \lvert t'_n-t_n \rvert >0$,
which contradicts $\lim_{n \to \infty}(t'_n-t_n)=0$.
\end{proof}

The proof of \eqref{eq:liminf} is rather complicated.
We assume to the contrary that
\begin{equation} \label{eq:contrary}
\inf_{t<\zeta} R(\xi(t), \slit(t)) > r
\end{equation}
holds for some $r>0$.

\begin{prop} \label{prop:slitconv}
Under the assumption~\eqref{eq:contrary},
the slit vector $\slit(t)$ converges to an element
$\slit(\zeta) \in \overline{\Slit} \subset \R^{3N}$ as $t \nearrow \zeta$.
\end{prop}

\begin{proof}
By \eqref{eq:contrary} and Lemma~\ref{lem:condB}, we have
\[
\sup_{t \in [0, \zeta)} \lvert b_{l}(\xi(t), \slit(t)) \rvert
\leq \sup_{(\xi_0, \slit) \in \R \times \Slit, R(\xi_0, \slit)>r}
\lvert b_{l}(\xi_0, \slit) \rvert < \infty.
\]
Hence the right-hand side of \eqref{eq:KLsvec} is integrable in $t$
over the interval $[0, \zeta)$.
\end{proof}

By Proposition~\ref{prop:slitconv},
the range $D_t:=D(\slit(t))$ converges to a domain $D_{\zeta}$
as $t \nearrow \zeta$ in the sense of kernel convergence, and
the limit domain $D_{\zeta}$ is of the form
$\uhp \setminus \bigcup_{j=1}^N C_{j, \zeta}$,
where $C_{j, \zeta}$ denotes the $j$-th ``slit'' corresponding to $\slit(\zeta)$.
The segment $C_{j, \zeta}$ may degenerate to a point
or be a subset of $\partial \uhp$ for some $j$.
Our goal is to show that actually $\slit(\zeta) \in \Slit$,
a contradiction to our assumption that $\zeta$ is the explosion time
of the solution $\slit(t)$ to the ODE \eqref{eq:KLsvec} on $\Slit$.

For this purpose, we extend the associated Komatu--Loewner evolution
$\{F_t\}_{t<\zeta}$ driven by $\xi(t)$ in $D$
continuously beyond $\zeta$ by regarding it as a Loewner evolution in $\uhp$
by means of \cite[Theorem~4.8]{Mu18}.
Let $\iota \colon D \hookrightarrow \uhp$ be the inclusion map and
$g^0_t \colon \uhp \setminus F_t \to \uhp$ be the canonical map
for $F_t$ in $\uhp$.
We define (by Schwarz's reflection)
\[
\iota_t := g^0_t \circ \iota \circ g_t^{-1}
\colon D_t \cup \Pi D_t \cup \partial \uhp \hookrightarrow \C.
\]
As explained in Section~\ref{sec:prel},
\cite[Theorem~2.6]{CFS17} or \cite[Theorem~4.8]{Mu18} implies that
$\{F_t\}_{t<\zeta}$ is produced by a generalized chordal Loewner equation
\begin{equation} \label{eq:gLo}
\frac{d}{dt}g^0_t(z)=2\pi \iota_t'(\xi(t))^2 \Psi_{\uhp}(g^0_t(z), \iota_t(\xi(t))),
\quad z \in \uhp.
\end{equation}
In other words, its half-plane capacity and driving function in $\uhp$ are given by
\begin{equation} \label{eq:transformed}
a^0_t:=\hcap^{\uhp}(F_t)=2\int_0^t\iota_s'(\xi(s))^2\,ds \quad \text{and}
\quad U(t):=\iota_t(\xi(t)),
\end{equation}
respectively.
The following three assertions hold under the assumption~\eqref{eq:contrary}:

\begin{prop} \label{prop:c1bound}
There exist an open interval $J$ and constants $t_1 \in (0, \zeta)$ and $A>1$
such that $\xi([t_1, \zeta]) \subset J$ and
\[
\frac{1}{2A} \leq \iota_t'(\xi_0) \leq \frac{3A}{2},
\quad \xi_0 \in J,\; t\in [t_1, \zeta).
\]
\end{prop}

\begin{cor} \label{cor:hcap}
The monotone limit $a^0_{\zeta-}:=\lim_{t \nearrow \zeta}a^0_t$ is finite.
\end{cor}

\begin{prop} \label{prop:drive}
The driving function $U(t)$ converges as $t \nearrow \zeta$.
\end{prop}

Corollary~\ref{cor:hcap} immediately follows
from \eqref{eq:transformed} and Proposition~\ref{prop:c1bound}.
The proof of Propositions~\ref{prop:c1bound} and \ref{prop:drive}
is postponed to Sections~\ref{subsec:claim_hcap} and \ref{subsec:claim_drive}.

We now put
\[
\check{g}^0_t:=g^0_{(a^0)^{-1}(2t)},\quad \check{F}_t:=F_{(a^0)^{-1}(2t)}
\quad \text{and} \quad \check{U}(t):=U((a^0)^{-1}(2t))
\]
for $0 \leq t < \check{\zeta}:=a^0_{\zeta-}/2<\infty$.
By this time-change, the equation~\eqref{eq:gLo} is reduced
to the usual Loewner equation \eqref{eq:Lo}, and
the evolution $(\check{g}^0_t, \check{F}_t)_{t<\check{\zeta}}$
is now produced by \eqref{eq:Lo} driven by $\check{U}$.
Since the driving function $\check{U}$ can be extended continuously
to the interval $[0, \infty)$ by Proposition~\ref{prop:drive},
we can extend $(\check{g}^0_t, \check{F}_t)_{t \in [0, \check{\zeta})}$
continuously to $[0, \infty)$
by solving \eqref{eq:Lo} driven by $\check{U}$ so extended.

\begin{prop} \label{prop:nonintersect}
Under the assumption~\eqref{eq:contrary},
the inclusion $\check{F}_{\check{\zeta}} \subset D$ holds.
In particular, the image
$\check{g}^0_{\check{\zeta}}(D \setminus \check{F}_{\check{\zeta}})
=\uhp \setminus \bigcup_{j=1}^N \check{g}^0_{\check{\zeta}}(C_j)$
is a non-degenerate $(N+1)$-connected domain.
(`Non-degenerate' means that none of the boundary components of
$\check{g}^0_{\check{\zeta}}(D \setminus \check{F}_{\check{\zeta}})$ is a singleton.)
\end{prop}

\begin{proof}
For $t \in [0, \zeta)$, we set
\begin{equation} \label{eq:iota_norm}
h^1_t(z):=\frac{\iota_t(rz+\xi(t))-\iota_t(\xi(t))}{r\iota_t'(\xi(t))}
=\frac{\iota_t(rz+\xi(t))-U(t)}{r\iota_t'(\xi(t))}.
\end{equation}
The function $h^1_t(z)$ is univalent and defined on a domain containing $\disk$
by the assumption~\eqref{eq:contrary}.
Thus, it belongs to $S$ by definition,
and Koebe's one-quarter theorem~\eqref{eq:Koebe} implies that
$\lvert \iota_t(rz+\xi(t))-U(t)\rvert\geq r\iota_t'(\xi(t))/4$ for $z\notin\disk$.
Combining this inequality with \eqref{eq:contrary} and Proposition~\ref{prop:c1bound}, we get
\[
\min_{1 \leq j \leq N}\dist(g^0_t(C_j), U(t)) \geq \frac{r}{8A}
\]
for $t < \zeta$.
By passing to the limit as $t \nearrow \zeta$, we have
\[
\min_{1 \leq j \leq N}\dist(\check{g}^0_{\check{\zeta}}(C_j),
\check{U}(\check{\zeta}))>0,
\]
which yields $\check{F}_{\check{\zeta}} \cap \bigcup_{j=1}^N C_j = \emptyset$
in view of \cite[Section~2.1]{LSW01}.
(See also \cite[Theorem~5.5~(i)]{CF18}, which we have already referred to
in Section~\ref{sec:prel}.)
$D \setminus \check{F}_{\check{\zeta}}$ is thus a non-degenerate
$(N+1)$-connected domain.
Since the non-degeneracy of a finitely multiply connected domain
is preserved under conformal maps (cf.\ \cite[Exercise~15.2.1]{Co95}),
the proposition follows.
\end{proof}

\begin{prop} \label{prop:nondegeneracy}
The slit domain $D_{\zeta}=\uhp \setminus \bigcup_{j=1}^{N}C_{j, \zeta}$
is non-degenerate and $(N+1)$-connected.
\end{prop}

\begin{proof}
We consider the two families of domains
\[
\check{D}_t:=D_{(a^0)^{-1}(2t)},\; t<\check{\zeta}
\quad \text{and} \quad
\check{D}^0_t:=\check{g}^0_t(D \setminus \check{F}_{t}),
\; t \leq \check{\zeta}.
\]
We have seen just after Proposition~\ref{prop:slitconv}
that the former family converges to
$D_{\zeta}$ as $t \nearrow \check{\zeta}$ in the sense of kernel convergence.
On the other hand, we can observe that
the latter one converges to $\check{D}^0_{\check{\zeta}}$ as follows:
$\check{g}^0_t$ converges to $\check{g}^0_{\check{\zeta}}$ locally uniformly
on $\uhp\setminus \check{F}_{\check{\zeta}}$ as $t\nearrow\check{\zeta}$,
since $\check{g}^0_t(z)$ is jointly continuous
in $(t,z)\in\bigcup_{s\in[0,\infty)}\{s\}\times(\uhp\setminus\check{F}_s)$
by a general theory of ODEs. In particular, the same convergence occurs
on a smaller domain $D\setminus\check{F}_{\check{\zeta}}$.
Moreover, $D\setminus\check{F}_t$ converges
to $D\setminus\check{F}_{\check{\zeta}}$ in the sense of kernel convergence
as $t\nearrow\check{\zeta}$, because the continuity of the hulls $\{\check{F}_t\}_t$
in $D$ is inherited from that in $\uhp$ (cf. \cite[Proposition~4.7]{Mu18}).
Thus, it follows from the implication
$\eqref{cond:fnconv}\Rightarrow\eqref{cond:ranconv}$ of Theorem~\ref{thm:kernel}
that $\check{D}^0_t$ converges to $\check{D}^0_{\check{\zeta}}$.

We now apply the implication
$\eqref{cond:ranconv}\Rightarrow\eqref{cond:fnconv}$ of Theorem~\ref{thm:kernel}
to the mappings $\iota_t^{-1}\colon\check{D}^0_t\to\check{D}_t$, $t<\zeta$.
Then we see that there exists a conformal map
$\iota_{\zeta}^{-1}\colon\check{D}^0_{\check{\zeta}}\to D_{\zeta}$,
which proves the proposition due to Proposition~\ref{prop:nonintersect}.
\end{proof}

The claim of Proposition~\ref{prop:nondegeneracy} is equivalent to
$\slit(\zeta) \in \Slit$, as was to be proven.

\subsection{Proof of Proposition~\ref{prop:c1bound}}
\label{subsec:claim_hcap}

The aim of this subsection is to prove Proposition~\ref{prop:c1bound}
under the assumption~\eqref{eq:contrary}.
By Proposition~\ref{prop:slitconv}, there is a constant $L>0$ so that
$\xi([0,\zeta]) \cup \bigcup_{t \in [0, \zeta]}\bigcup_{j=1}^N C_{j,t} \subset B(0,L)$,
where $C_{j, t}:=C_j(\slit(t))$.
Since the conformal map $\iota_t$ is the composite of three maps
hydrodynamically normalized, it satisfies
\[
\iota_{t}(z) = z + \frac{c_{t}}{z} + o(z^{-1})\; (z \to \infty),\quad z \in \Delta(0, L),
\]
for some constant $c_{t}$.
We define a normalized function $f_t$ on $\disk^*:=\Delta(0,1)$
by $f_t(z):=L^{-1}\iota_t(Lz)$.
The function $f_t$ is an element of the set
\[
\Sigma:=\{f \colon \disk^* \to \C;
\text{$f$ is univalent}, f(\infty)=\infty\; \text{and}\; \res(f, \infty)=1\}.
\]
Hence we have $\C \setminus f_t(\disk^*) \subset \overline{B(0,2)}$
by \cite[Lemma~3.5]{Mu18}.
In terms of $\iota_t$, this means that
\begin{equation} \label{eq:locbdd}
\iota_t(B(0, L)) \subset \C \setminus \iota_t(\Delta(0, L))
\subset \overline{B(0,2L)}.
\end{equation}

If $D_t$ had no slits,
then the boundedness of $\iota_t'(\xi(t))$ would follow from \eqref{eq:locbdd}
combined with elementary tools in complex analysis such as Schwarz's lemma.
These tools, however, do not work on multiply connected domains.
For this reason, we employ the boundary Harnack principle instead:

\begin{prop}[{\cite[Theorem~8.7.14]{AG01}}] \label{prop:BHP}
Let $G \subset \R^2$ be a bounded Lipschitz domain,
$V \subset \R^2$ be an open set,
$K$ be a compact subset of $V$, and $z_0 \in G$.
Then there exists a constant $A>1$ such that,
for any two harmonic functions $h_1$ and $h_2$ on $G$
taking value zero on $V \cap \partial G$, it holds that
\[
\frac{h_1(x)}{h_2(x)}\frac{h_2(z_0)}{h_1(z_0)} \leq A, \quad x \in K \cap G.
\]
\end{prop}

\begin{figure}
\centering
\includegraphics[width=0.9\columnwidth]{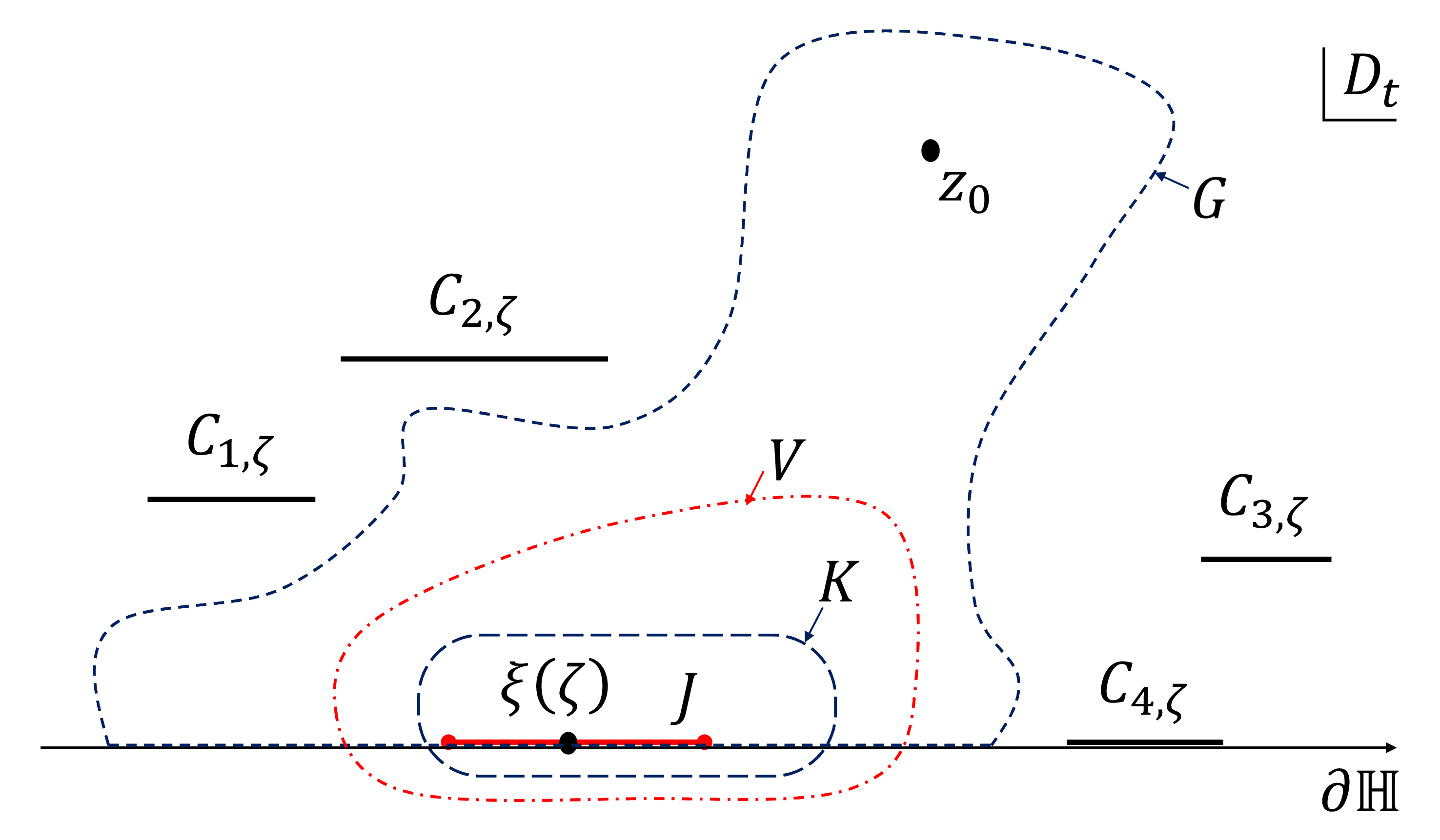}
\caption{The way to choose $J$, $K$, $V$, $G$ and $z_0$}
\label{fig:BHP}
\end{figure}

We shall apply this proposition to the harmonic functions
$h_1(z)=\Im \iota_t(z)$ and $h_2(z)=\Im z$.
The sets $G$, $V$, $K$ and point $z_0$ in the assumption are chosen as follows
(see Figure~\ref{fig:BHP}):
By the assumption \eqref{eq:contrary} and Proposition~\ref{prop:slitconv},
there exist a constant $t_1 \in [0, \zeta)$ and
finite open subinterval $J$ of $\partial \uhp$ such that
\[
\xi([t_1, \zeta]) \subset J \quad \text{and} \quad
\overline{J} \cap \bigcup_{t \in [t_1, \zeta]} C_{j, t} = \emptyset.
\]
For this interval $J$, there exist a relatively compact open set $O$ and
an open set $V$ such that
\[
J \subset O \subset \overline{O} \subset V \quad \text{and} \quad
\overline{V} \cap \bigcup_{t \in [t_1, \zeta]} C_{j, t}=\emptyset.
\]
For this set $V$ and an arbitrary fixed point $z_0 \in D$ with $\Im z_0 \geq 6L$,
we can take a bounded domain $G \subset D$ with smooth boundary
so that
\[
z_0 \in G,\quad V \cap \partial \uhp \subset \partial G \quad \text{and}
\quad G \cap \bigcup_{t \in [t_1, \zeta]} C_{j, t}=\emptyset.
\]
Now we apply Proposition~\ref{prop:BHP} to $h_1$ and $h_2$
with $G$, $V$, $K:=\overline{O}$ and $z_0$ chosen in this way
to obtain
\begin{equation} \label{eq:z0BHP}
A^{-1} \leq \frac{\Im \iota_t(z)}{\Im z}\frac{\Im z_{0}}{\Im \iota_t(z_{0})} \leq A,
\quad z \in G \cap K,\ t \in [t_1, \zeta),
\end{equation}
for a constant $A > 1$ independent of $z$ and $t$.

On the other hand, we can observe from \eqref{eq:locbdd} that
\begin{align*}
\left\lvert \Im \iota_t(z_0) - \Im z_0 \right\rvert
&= \left\lvert \mean{\uhp}{z_0}{\Im \iota_t(Z^{\uhp}_{\sigma_{\mathcal{C}_t}})
- \Im Z^{\uhp}_{\sigma_{\mathcal{C}_t}};
\sigma_{\mathcal{C}_t} < \infty} \right\rvert \\
&\leq \mean{\uhp}{z_0}{\Im \iota_t(Z^{\uhp}_{\sigma_{\mathcal{C}_t}})
+ \Im Z^{\uhp}_{\sigma_{\mathcal{C}_t}}; \sigma_{\mathcal{C}_t} < \infty} \leq 3L.
\end{align*}
Here, $Z^{\uhp}$ is an absorbing Brownian motion in $\uhp$,
$\sigma_{\mathcal{C}_t}$ is the hitting time of $Z^{\uhp}$
to $\mathcal{C}_t:=\bigcup_{j=1}^N C_{j,t}$, and
$\mathbb{E}^{\uhp}_{z_0}$ stands for the expectation
with respect to $Z^{\uhp}$ starting at $z_0$.
Hence we have
\begin{equation} \label{eq:z0bound}
\left\lvert \frac{\Im \iota_t(z_{0})}{\Im z_{0}} - 1 \right\rvert
\leq \frac{3L}{\Im z_{0}} \leq \frac{1}{2},
\quad \text{i.e.,} \quad
\frac{1}{2} \leq \frac{\Im \iota_{t}(z_{0})}{\Im z_{0}} \leq \frac{3}{2}.
\end{equation}
Substituting \eqref{eq:z0bound} into \eqref{eq:z0BHP} yields that
\begin{equation} \label{eq:c0bound}
\frac{1}{2A} \leq \frac{\Im \iota_t(z)}{\Im z} \leq \frac{3A}{2},
\quad z \in G \cap K,\ t \in [t_1, \zeta).
\end{equation}
Since the function $\iota_t$ is defined across $\partial \uhp$
by Schwarz's reflection,
it is easily checked that $\lim_{z \to \xi_0} \Im \iota_t(z)/\Im z = \iota_t'(\xi_0)$
for $\xi_0 \in \partial \uhp$.
Thus by taking the limit as $z$ goes to $\xi_0 \in J$ in \eqref{eq:c0bound}, we have
\[
\frac{1}{2A} \leq \iota_t'(\xi_0) \leq \frac{3A}{2}, \quad \xi_0 \in J,
\]
which proves Proposition~\ref{prop:c1bound}.

\subsection{Proof of Proposition~\ref{prop:drive}}
\label{subsec:claim_drive}

The aim of this subsection
is to prove Proposition~\ref{prop:drive} under the assumption~\eqref{eq:contrary}.
To this end, we approximate the continuous function $\xi$
by $\xi^{\varepsilon} \in C^{1}[0, \infty)$ so that
\[
\sup_{t \in [t_1, \zeta]}\lvert \xi(t) - \xi^{\varepsilon}(t) \rvert <\varepsilon
\quad \text{and}\quad \{\xi^{\varepsilon}(t); t \in [t_1, \zeta]\} \subset J
\]
hold for $\varepsilon \in (0, r/2)$.
Here, the constant $t_1$ and interval $J$ are
those in Proposition~\ref{prop:c1bound}.

\begin{lem} \label{lem:approx}
$\iota_t(\xi^{\varepsilon}(t))$ converges as $t \nearrow \zeta$
for each fixed $\varepsilon \in (0, r/2)$.
\end{lem}

\begin{proof}
$\iota_t(\xi^{\varepsilon}(t))$ is represented as
\begin{align*}
\iota_t(\xi^{\varepsilon}(t)) &= \iota_{t_1}(\xi^{\varepsilon}(t_1))
+ \int_{t_1}^{t} \frac{d}{ds}\iota_s(\xi^{\varepsilon}(s))\,ds \\
&= \iota_{t_1}(\xi^{\varepsilon}(t_1))
+ \int_{t_1}^{t} \left\{ (\partial_s \iota_s)(\xi^{\varepsilon}(s))
	+ \iota_s'(\xi^{\varepsilon}(s))\dot{\xi^{\varepsilon}}(s) \right\}\,ds
\end{align*}
for $t \in [t_1, \zeta)$.
By Proposition~\ref{prop:c1bound}, we have
\[
\sup_{s \in [t_1, \zeta)}
\lvert \iota_s'(\xi^{\varepsilon}(s))\dot{\xi^{\varepsilon}}(s) \rvert
\leq \frac{3A}{2} \max_{s \in [t_1, \zeta]} \lvert \dot{\xi^{\varepsilon}}(s) \rvert
< \infty.
\]
Thus it suffices to prove that
$\sup_{s \in [t_1, \zeta)}\lvert(\partial_s \iota_s)(\xi^{\varepsilon}(s))\rvert<\infty$
in order to establish the lemma.

We begin with the computation of $\partial_t \iota_t(z)$ for $z \in D$.
By the definition of $\iota_t$ and $\mathbf{H}_{\slit(t)}$, we have
\begin{align}
\partial_t\iota_t(z)&=(\partial_tg^0_t)(\iota \circ g_t^{-1}(z))
	+ (g^0_t)'(\iota \circ g_t^{-1}(z))\partial_tg_t^{-1}(z) \notag \\
&= \frac{2\iota_t(\xi(t))^2}{\iota_t(z) - U(t)} - (g^0_t)'(\iota \circ g_t^{-1}(z))
	(g_t^{-1})'(z)(\partial_tg_t)(g_t^{-1}(z)) \notag \\
&= \frac{2\iota_t(\xi(t))^2}{\iota_t(z) - \iota_t(\xi(t))}
	+ 2\pi \iota_t'(z)\Psi_{\slit(t)}(z, \xi(t)) \notag \\
&= \frac{2\iota_t(\xi(t))^2}{\iota_t(z) - \iota_t(\xi(t))}
	- \frac{2\iota_t'(z)}{z - \xi(t)}
	+ 2\pi \iota_t'(z) \mathbf{H}_{\slit(t)}(z, \xi(t)). \label{eq:tderiviota}
\end{align}
We denote the first two terms in the last expression~\eqref{eq:tderiviota}
by $\Theta_t(z)$. Since $\iota_t$ is holomorphic on the disk $B(\xi(t), r)$,
so is $\Theta_t$ on the punctured disk $B(\xi(t), r) \setminus \{\xi(t)\}$.
Actually, $\xi(t)$ is a removable singularity of $\Theta_t$ because
\[
\Theta_t(z)=
\frac{2\iota_t(\xi(t))^{2}}{\iota_t(z) - \iota_t(\xi(t))} - \frac{2\iota_t'(z)}{z - \xi(t)}
\to -3\iota_t''(\xi(t)), \quad z \to \xi(t),
\]
by \cite[Proposition~4.40]{La05}.
Consequently, the identity
$\partial_t\iota_t(z)=\Theta_t(z)+2\pi\iota'_t(z)\mathbf{H}_{\slit(t)}(z,\xi)$ is valid
for all $z \in D \cup \Pi D \cup \partial \uhp$.

We now give a closer look at $\Theta_t(z)$.
Since the function $2h^1_t(z/2)$ with $h^1_t$ defined by \eqref{eq:iota_norm}
belongs to $S$, we have
\begin{equation} \label{eq:denominator}
\lvert \iota_t(z) - \iota_t(\xi(t)) \rvert \geq \frac{r\iota_t'(\xi(t))}{8}
\geq \frac{r}{16A},
\quad z \in \partial B\left(\xi(t), \frac{r}{2}\right),\; t \in [t_1, \zeta),
\end{equation}
by Proposition~\ref{prop:c1bound} and
Koebe's one-quarter theorem~\eqref{eq:Koebe}.
Moreover, we utilize the distortion theorem
(see \cite[Theorem~14.7.9~(a)]{Co95} or \cite[Theorem~1.6~(11)]{Po75}):
\begin{equation} \label{eq:distortion}
\frac{1-\lvert z \rvert}{(1+\lvert z \rvert)^3}
\leq \lvert f'(z) \rvert
\leq \frac{1+\lvert z \rvert}{(1-\lvert z \rvert)^3}, \quad z \in \disk,\; f \in S.
\end{equation}
The inequality \eqref{eq:distortion} with $f=h^1_t$
and Proposition~\ref{prop:c1bound} yield,
for $z \in B(\xi(t), r)$ and $t \in [t_1, \zeta)$,
\begin{equation} \label{eq:numerator}
\lvert \iota_t'(z) \rvert \leq r\iota_t'(\xi(t))
	\frac{1+r^{-1}\lvert z-\xi(t) \rvert}{(1-r^{-1}\lvert z-\xi(t) \rvert)^3}
\leq \frac{3A}{2}\frac{r^4+r^3\lvert z-\xi(t) \rvert}{(r-\lvert z-\xi(t) \rvert)^3}.
\end{equation}
In addition, it follows from \eqref{eq:locbdd} that
\begin{equation} \label{eq:iotabound}
\sup_{t \in [t_1, \zeta)}\lvert \iota_t(\xi(t)) \rvert \leq 2L \quad \text{and} \quad
\sup_{t \in [t_1, \zeta)}\lvert \iota_t(\xi^{\varepsilon}(t)) \rvert \leq 2L.
\end{equation}
By \eqref{eq:denominator}, \eqref{eq:numerator} and \eqref{eq:iotabound},
there exists a constant $M_1>0$ such that
\[
\sup_{t \in [t_1, \zeta)}\max_{z \in \partial B(\xi(t), r/2)}
\lvert \Theta_t(z) \rvert \leq M_1.
\]
The maximal value principle for $\Theta_t$ then implies that
\[
\sup_{t \in [t_1, \zeta)}\sup_{z \in B(\xi(t), r/2)} \lvert \Theta_t(z) \rvert \leq M_1.
\]
Hence it holds that
\begin{equation} \label{eq:Ft}
\sup_{s \in [t_1, \zeta)} \lvert \Theta_s(\xi^{\varepsilon}(s)) \vert \leq M_1.
\end{equation}

It remains to estimate
\[
(\partial_s \iota_s)(\xi^{\varepsilon}(s))-\Theta_s(\xi^{\varepsilon}(s))
=2\pi \iota_s(\xi^{\varepsilon}(s))\mathbf{H}_{\slit(s)}(\xi^{\varepsilon}(s), \xi(s)).
\]
By \eqref{eq:inclBMD} and \eqref{eq:contrary}, we have
\[
\lvert \mathbf{H}_{\slit(t)}(z, \xi(t)) \rvert
=\left\lvert \Psi_{\slit(t)}(z, \xi(t)) + \frac{1}{\pi}\frac{1}{z-\xi(t)} \right\rvert
\leq \frac{5}{4\pi r}
\]
for $z \in \partial B(\xi(t), r)$ and $t \in [0, \zeta)$.
By the maximal value principle for $\mathbf{H}_{\slit(t)}(\cdot, \xi(t))$, we obtain
\begin{equation} \label{eq:Hslit}
\sup_{t \in [0, \zeta)}\sup_{z \in B(\xi(t), r)}
\lvert \mathbf{H}_{\slit(t)}(z, \xi(t)) \rvert \leq \frac{5}{4\pi r}.
\end{equation}
\eqref{eq:iotabound} and \eqref{eq:Hslit} yield
\begin{equation} \label{eq:rem_term}
\sup_{s \in [t_1, \zeta)} \lvert
2\pi \iota_s(\xi^{\varepsilon}(s))\mathbf{H}_{\slit(s)}(\xi^{\varepsilon}(s), \xi(s))
\rvert \leq \frac{5L}{r}.
\end{equation}
It follows from \eqref{eq:tderiviota}, \eqref{eq:Ft} and \eqref{eq:rem_term} that
$\sup_{s \in [t_1, \zeta)}\lvert(\partial_s \iota_s)(\xi^{\varepsilon}(s))\rvert<\infty$,
which is the desired conclusion.
\end{proof}

Recall that
$\sup_{t \in [t_1, \zeta]}\lvert \xi(t) - \xi^{\varepsilon}(t) \rvert <\varepsilon$
is assumed at the beginning of this subsection.
It holds that
\begin{align*}
&\limsup_{t \nearrow \zeta}\iota_t(\xi(t)) - \liminf_{t \nearrow \zeta}\iota_t(\xi(t))\\
&\leq \lvert \limsup_{t \nearrow \zeta}(\iota_t(\xi(t))
	- \iota_t(\xi^{\varepsilon}(t))) \rvert
	- \lvert \liminf_{t \nearrow \zeta}(\iota_t(\xi(t))
	- \iota_t(\xi^{\varepsilon}(t)) \rvert \\
&\leq \frac{3A}{2}\varepsilon + \frac{3A}{2}\varepsilon = 3A\varepsilon
\end{align*}
by Lemma~\ref{lem:approx} and Proposition~\ref{prop:c1bound}.
By letting $\varepsilon \to 0$ in this inequality
and taking \eqref{eq:iotabound} into account,
we observe that $U(t)=\iota_t(\xi(t))$ converges as $t \nearrow \zeta$.
The proof of Proposition~\ref{prop:drive} and thus
of Theorem~\ref{thm:char_d}~\eqref{thm:char_dI} is now complete.

\subsection{Proof of Theorem~\ref{thm:char_s}~\eqref{thm:char_sI}}
\label{subsec:char_s}

This subsection is devoted to the proof
of Theorem~\ref{thm:char_s}~\eqref{thm:char_sI},
which proceeds along lines similar to those in Section~\ref{subsec:outline}.
Suppose that functions $\alpha \geq 0$ and $b$ on $\Slit$ satisfy
the assumption of Theorem~\ref{thm:char_s}.
We denote by $\mathbb{P}_{\mathbf{w}}$ the law of the solution
$W_t=(\xi(t), \slit(t))$ to the SDEs~\eqref{eq:KLsvec} and \eqref{eq:SKLE}
with initial value $W_0=\mathbf{w} \in \R \times \Slit$.
We write $\mathbb{P}_{\mathbf{w}^{\mathrm{int}}}$ simply as $\mathbb{P}$.
As mentioned in Chapter~IV, Section~6 of \cite{IW89},
the solution $W=(W_t, \mathbb{P}_{\mathbf{w}})$ becomes a diffusion process
on the state space
$(\R \times \Slit)_{\infty}:=(\R \times \Slit) \cup \{\mathbf{w}_{\infty}\}$,
where $\mathbf{w}_{\infty}$ is the cemetery,
with respect to the augmented filtration $(\mathcal{F}_t)_{t \geq 0}$
of the Brownian motion $(B_t)_{t \geq 0}$ in \eqref{eq:SKLE}.
We denote the lifetime of $W$ by $\zeta$.
(This is a slight abuse of notation,
but there should be no risk of confusion.)

We define an operator
$\Lambda_r \colon \C^{\R \times \Slit} \to \C^{(\R \times \Slit)_{\infty}}$
for $r>0$ by
\[
\Lambda_r f(\mathbf{w}):=\begin{cases}
f(\mathbf{w})
&\text{if $\mathbf{w} \in \R \times \Slit$ and $R(\mathbf{w}) \geq r$} \\
0 &\text{otherwise.}
\end{cases}
\]
Using this operator, we define a process $W^r_t=(\xi^r(t), \slit^r(t))$ by
\[
W^r_t:=W_0+\int_0^t \Lambda_r \alpha(W_s)\,dB_s
+\int_0^t \Lambda_r b(W_s)\,ds,\quad t \geq 0.
\]
The functions $\Lambda_r \alpha$ and $\Lambda_r b$ are bounded
by Condition~(B).
Hence $(W^r_t)_{t \geq 0}$ is a continuous semimartingale
whose local martingale part is a square-integrable martingale.
Let $\tau_r:=\inf\{t>0; W_t=\mathbf{w}_{\infty}\; \text{or}\; R(W_t)<r\}$.

\begin{prop} \label{prop:modif_W}
For any starting point $\mathbf{w} \in \R \times \Slit$ and
$r \in (0, R(\mathbf{w}))$, it holds that $W_t=W^r_t$ for all $t \in [0, \tau_r)$
$\mathbb{P}_{\mathbf{w}}$-almost surely.
In particular, $W_t$ converges in $\overline{\R \times \Slit}$
as $t \nearrow \tau_r$
$\mathbb{P}_{\mathbf{w}}$-almost everywhere on $\{\tau_r<\infty\}$.
\end{prop}

\begin{proof}
Since $\alpha(W_t)=\Lambda_r \alpha(W_t)$ and $b(W_t)=\Lambda_r b(W_t)$
hold for $t<\tau_r$,
the conclusion follows from \cite[Proposition~II.2.2~(iv)]{IW89}
and the localization by an appropriate sequence of stopping times.
\end{proof}

For $r>0$, we define stopping times $\{\tau_{r,n}\}_{n=0}^{\infty}$,
$\{\tau'_{r,n}\}_{n=0}^{\infty}$ and $\{\sigma_{r,n}\}_{n=0}^{\infty}$ recursively by
$\tau'_{r,0}:=0$ and
\begin{align*}
\tau_{r,n}&:=\inf\{t>\tau'_{r,n}; W_t=\mathbf{w}_{\infty} \; \text{or}\; R(W_t) < r\}, \\
\sigma_{r,n}&:=\inf\{t>\tau'_{r,n}; W_t=\mathbf{w}_{\infty}
\; \text{or}\; \lvert \xi(t)-\xi(\tau'_{r,n}) \rvert \geq r\}, \\
\tau'_{r,n+1}&:=\inf\{t>\tau_{r,n}; W_t=\mathbf{w}_{\infty}
\; \text{or}\; R(W_t) \geq 4r\},
\end{align*}
and events $E_r$ and $E_{r,n}$, $n=1,2,\ldots$, by
\begin{align*}
E_r&:=\{\zeta<r^{-1},\; \liminf_{t \nearrow \zeta}R(W_t)=0
\; \text{and} \; \limsup_{t \nearrow \zeta}R(W_t) \geq 5r\}, \\
E_{r,n}&:=E_r \cap \{\sigma_{r,n} < \tau_{r,n}\}.
\end{align*}
Here we adopt the convention that $\inf\emptyset:=\infty$.
By definition, we have $\tau_{r,0}=\tau_r$ and
$\tau'_{r,n}(\omega)<\tau_{r,n}(\omega)<\tau'_{r, n+1}(\omega)<\zeta(\omega)<1/r$
for all $\omega \in E_r$ and $n \in \N$.

\begin{lem} \label{lem:limsup_rep}
It holds that $E_r = \liminf_n E_{r,n} = \limsup_n E_{r,n}$.
\end{lem}

\begin{proof}
Assume that there are a sample $\omega \in E_r$ and
an increasing sequence $\{n_k(\omega)\}_{k=1}^{\infty}$ of natural numbers
such that $\sigma_{r, n_k(\omega)}(\omega) \geq \tau_{r, n_k(\omega)}(\omega)$
holds for all $k$.
It follows from definition that
$\tau_{r,n_k(\omega)}(\omega)-\tau'_{r,n_k(\omega)}(\omega) \geq r/M$,
where $M$ is the constant in the proof of Proposition~\ref{prop:inf_sup}.
By this inequality, however, we have
\[
r^{-1} > \zeta(\omega)
> \sum_{k=1}^{\infty}(\tau_{r,n_k(\omega)}(\omega)-\tau'_{r,n_k(\omega)}(\omega))
=\infty,
\]
a contradiction.
Therefore, it holds that $E_r \subset \liminf_n E_{r,n}$.
Since it is obvious that $\liminf_n E_{r,n} \subset \limsup_n E_{r,n} \subset E_r$,
the lemma follows.
\end{proof}

\begin{prop} \label{prop:inf_sup_null}
The event
\[
E:=\{\zeta<\infty,\; \liminf_{t \nearrow \zeta}R(W_t)=0
\; \text{and} \; \limsup_{t \nearrow \zeta}R(W_t)>0\}
\]
is a $\mathbb{P}$-null set.
\end{prop}

\begin{proof}
We fix an arbitrary $r \in (0, R(\mathbf{w}^{\mathrm{int}}))$.
It follows from the strong Markov property of $W$ that
\begin{align}
&r^2\prob{}{}{E_{r,n}}
=\mean{}{}{(\xi(\sigma_{r,n})-\xi(\tau'_{r,n}))^2 \mathbf{1}_{E_{r,n}}} \notag \\
&\leq \mean{}{}{(\xi(\sigma_{r,n})-\xi(\tau'_{r,n}))^2
\mathbf{1}_{E \cap \{\sigma_{r,n}<\tau_{r,n}\}
\cap \{\sum_{k \geq n}(\tau_{r,k}-\tau'_{r,k})<r^{-1}\}}
\mathbf{1}_{\{\tau'_{r,n}<r^{-1}\wedge\zeta\}}} \notag \\
&=\mean{}{}{
\mean{}{W_{\tau'_{r,n}}}{(\xi(\sigma_{r,0})-\xi(0))^2
\mathbf{1}_{E \cap \{\sigma_{r,0}<\tau_r\}
\cap \{\sum_{k \geq 0}(\tau_{r,k}-\tau'_{r,k})<r^{-1}\}}}
\mathbf{1}_{\{\tau'_{r,n}<r^{-1}\wedge\zeta\}}} \notag \\
&=\mean{}{}{
\mean{}{W_{\tau'_{r,n}}}{(\xi^r(r^{-1}\wedge\sigma_{r,0}\wedge\tau_r)-\xi(0))^2}
\mathbf{1}_{\{\tau'_{r,n}<r^{-1}\wedge\zeta\}}} \label{eq:smp_W}.
\end{align}
Moreover, we have
\begin{align}
&\mean{}{W_{\tau'_{r,n}}}{(\xi^r(r^{-1}\wedge\sigma_{r,0}\wedge\tau_r)-\xi(0))^2}
\notag \\
&=\mean{}{W_{\tau'_{r,n}}}{
\left(\int_0^{r^{-1}\wedge\sigma_{r,0}\wedge\tau_r} \Lambda_r \alpha(W_s)\,dB_s
+\int_0^{r^{-1}\wedge\sigma_{r,0}\wedge\tau_r} \Lambda_r b(W_s)\,ds\right)^2}
\notag \\
&\leq 2\mean{}{W_{\tau'_{r,n}}}{
\left(\int_0^{r^{-1}\wedge\sigma_{r,0}\wedge\tau_r}
\Lambda_r \alpha(W_s)\,dB_s\right)^2
+\left(\int_0^{r^{-1}\wedge\sigma_{r,0}\wedge\tau_r} \Lambda_r b(W_s)\,ds\right)^2}
\notag \\
&\leq 2\mean{}{W_{\tau'_{r,n}}}{
\int_0^{r^{-1}\wedge\sigma_{r,0}\wedge\tau_r}\Lambda_r \alpha(W_s)^2\,ds
+r^{-1}\int_0^{r^{-1}\wedge\sigma_{r,0}\wedge\tau_r} \Lambda_r b(W_s)^2\,ds}
\notag \\
&\leq 2(1+r^{-1})M_r\mean{}{W_{\tau'_{r,n}}}{r^{-1}\wedge\sigma_{r,0}\wedge\tau_r},
\label{eq:xi_tau}
\end{align}
where $M_r:=\sup_{\mathbf{w}}(\Lambda_r\alpha(\mathbf{w})^2
\vee\Lambda_r b(\mathbf{w})^2)<\infty$.
Substituting \eqref{eq:xi_tau} into \eqref{eq:smp_W} yields
\begin{align*}
r^2\prob{}{}{E_{r,n}}
&\leq 2(1+r^{-1})M_r\mean{}{}{
\mean{}{W_{\tau'_{r,n}}}{r^{-1}\wedge\sigma_{r,0}\wedge\tau_r}
\mathbf{1}_{\{\tau'_{r,n}<r^{-1}\wedge\zeta\}}} \\
&=2(1+r^{-1})M_r\mean{}{}{
((r^{-1}+\tau'_{r,n})\wedge\sigma_{r,n}\wedge\tau_{r,n}-\tau'_{r,n})
\mathbf{1}_{\{\tau'_{r,n}<r^{-1}\wedge\zeta\}}} \\
&\leq 2(1+r^{-1})M_r\mean{}{}{
(2r^{-1})\wedge\sigma_{r,n}\wedge\tau_{r,n}-(2r^{-1})\wedge\tau'_{r,n}}.
\end{align*}
Hence we have
\begin{align*}
\sum_{n=0}^{\infty}\prob{}{}{E_{r,n}}
&\leq 2r^{-2}(1+r^{-1})M_r \sum_{n=0}^{\infty}
\mean{}{}{(2r^{-1})\wedge\sigma_{r,n}\wedge\tau_{r,n}-(2r^{-1})\wedge\tau'_{r,n}} \\
&\leq 4r^{-3}(1+r^{-1})M_r < \infty.
\end{align*}
It follows from the first Borel--Cantelli lemma that
$\prob{}{}{\limsup_n E_{r,n}}=0$,
which implies $\prob{}{}{E_r}=0$ by Lemma~\ref{lem:limsup_rep}.
Since $E=\bigcup_k E_{1/k}$ holds, we obtain $\prob{}{}{E}=0$.
\end{proof}

By Proposition~\ref{prop:inf_sup_null},
we can establish Theorem~\ref{thm:char_s}~\eqref{thm:char_sI}
if we prove that the event
\[
E':=\{\zeta<\infty,\; \liminf_{t \nearrow \zeta}R(W_t)>0\}
=\{\zeta<\infty,\; \inf_{t < \zeta}R(W_t)>0\}
\]
is a $\mathbb{P}$-null set.
To do this,
we denote by $\{F_t\}_{t<\zeta}$ the $\skle_{\alpha, b}$ driven by $\xi(t)$
and take over the notations in Section~\ref{subsec:outline}
such as $g^0_t$, $\iota_t$ and so on.
The relation~\eqref{eq:transformed} is vaild also in this case.
For a moment, we fix a constant $r \in (0, R(\mathbf{w}^{\mathrm{int}}))$.

\begin{prop} \label{prop:rad_c1bound}
There exist a random open interval $J(\omega)$
and constants $t_1(\omega) \in (0, \tau_r(\omega))$ and $A(\omega)>1$
such that $\xi([t_1, \tau_r]) \subset J$ and
\[
\frac{1}{2A} \leq \iota_t'(\xi_0) \leq \frac{3A}{2},
\quad \xi_0 \in J,\; t\in [t_1, \tau_r),
\]
hold $\mathbb{P}$-almost everywhere on $\{\tau_r<\infty\}$.
\end{prop}

\begin{proof}
For $\mathbb{P}$-a.a.\ $\omega \in \{\zeta=\tau_r<\infty\}$,
it holds that $\inf_{t<\zeta(\omega)}R(W_t(\omega)) \geq r$.
Hence the conclusion follows from Propositions~\ref{prop:modif_W}
and \ref{prop:c1bound}.
For $\omega \in \{\tau_r<\zeta\}$, the conclusion is trivial.
\end{proof}

\begin{cor} \label{cor:rad_hcap}
The monotone limit $a^0_{\tau_r-}:=\lim_{t \nearrow \tau_r}a^0_t$ is finite
$\mathbb{P}$-almost everywhere on $\{\tau_r<\infty\}$.
\end{cor}

\begin{prop} \label{prop:rad_drive}
The process $U(t)=\iota_t(\xi(t))$ converges as $t \nearrow \tau_r$
$\mathbb{P}$-almost everywhere on $\{\tau_r<\infty\}$.
\end{prop}

\begin{proof}
While this proposition follows from Proposition~\ref{prop:drive},
we can give a shorter proof in this case by using It\^o's formula.
By \cite[Theorem~2.8]{CFS17} or \cite[Eq.~(4.7)]{Mu18}, it holds that
\begin{align}
U(t)&=\xi^{\mathrm{int}}+\int_0^t\iota_s'(\xi(s))\alpha(W_s)\,dB_s
+\int_0^t\iota_s'(\xi(s))\left(b_{\bmd}(W_s)-b(W_s)\right)\,dt \notag \\
&\phantom{=}{}+\frac{1}{2}\int_0^t\iota_s''(\xi(s))\left(\alpha(W_s)^2-6\right)\,dt
\label{eq:Ito_drive}
\end{align}
for $t<\zeta$ almost surely under $\mathbb{P}$.
Here, $b_{\bmd}$ is the BMD domain constant appearing in Lemma~\ref{lem:condB}.

We set
\[
\mu_t:=\iota_t'(\xi(t))\mathbf{1}_{\{t<\tau_r\}}\quad \text{and}\quad
\nu_t:=\iota_t''(\xi(t))\mathbf{1}_{\{t<\tau_r\}}.
\]
By Proposition~\ref{prop:rad_c1bound},
we can regard $\mu_t$ as a progressively measurable process
that is bounded on every compact subinterval of $[0, \infty)$ a.s.
We apply Bieberbach's theorem to the function $h^1_t$
defined in \eqref{eq:iota_norm} and use Proposition~\ref{prop:rad_c1bound} again
to obtain
\[
\frac{r\lvert \iota_t''(\xi(t)) \rvert}{2\iota_t'(\xi(t))} \leq 2, \quad
\text{i.e.,} \quad \lvert \iota_t''(\xi(t)) \rvert \leq \frac{4}{r}\iota_t'(\xi(t))
\leq \frac{3A}{2r}
\]
for $t<\tau_r$ a.e.\ on $\{\tau_r<\infty\}$.
Hence $\nu_t$ is also progressively measurable and
bounded on every compact subinterval of $[0, \infty)$ a.s.
In this way, we observe that
\[
\int_0^t \left\{ (\mu_t \Lambda_r \alpha(W_s))^2
+\lvert \mu_t \Lambda_r(b_{\bmd}-b)(W_s) \rvert
+\lvert \nu_t (\Lambda_r \alpha(W_s)^2-6) \rvert \right\}\,ds<\infty
\]
for all $t \in [0, \infty)$ a.s., 
which implies that a process
\begin{align*}
U^r(t)&:=\xi^{\mathrm{int}}+\int_0^t\mu_s\Lambda_r\alpha(W_s)\,dB_s
+\int_0^t\mu_s\Lambda_r(b_{\bmd}-b)(W_s)\,dt \\
&\phantom{:=}{}+\frac{1}{2}\int_0^t\nu_s\left(\Lambda_r\alpha(W_s)^2-6\right)\,dt
\end{align*}
is a continuous semimartingale on $[0, \infty)$.
We can check in a way similar to the proof of Proposition~\ref{prop:modif_W}
that $U(t)=U^r(t)$ holds for all $t<\tau_r$ a.s.
In particular, $U(t)$ converges as $t \nearrow \tau_r$ on $\{\tau_r<\infty\}$.
\end{proof}

Let $E'_r:=\{\zeta<\infty,\; \inf_{t<\zeta}R(W_t) \geq r\}$.
It holds that $\tau_r=\zeta<\infty$ on $E'_r$.
From Propositions~\ref{prop:rad_c1bound}, \ref{prop:rad_drive} and
Corollary~\ref{cor:rad_hcap}, it follows that
Propositions~\ref{prop:nonintersect} and \ref{prop:nondegeneracy} hold
for $\mathbb{P}$-a.a.\ $\omega \in E'_r$.
Hence we have $W^r_{\zeta(\omega)}(\omega) \in \R \times \Slit$
for $\mathbb{P}$-a.a.\ $\omega \in E'_r$,
which yields $\prob{}{}{E'_r}=0$ by the definition of $\zeta$.
Since $E'=\bigcup_k E'_{1/k}$ holds, we have $\prob{}{}{E'}=0$,
which finishes the proof of Theorem~\ref{thm:char_s}~\eqref{thm:char_sI}.

\section*{Acknowledgements}

I wish to express my gratitude to Professor Roland M.\ Friedrich
for pointing out a lack of references in Section~\ref{sec:intro}
and to the anonymous referee for his or her suggestions
very helpful in making the proof transparent.

\end{document}